\documentclass[11pt]{amsart}
\usepackage{indentfirst,latexsym,bm,color}
\usepackage{amsmath,amssymb}
\usepackage{amsthm}
\usepackage{fancyhdr}
\usepackage{graphics}
\usepackage{indentfirst,latexsym,bm,amsthm,amssymb,graphicx}
\usepackage{colortbl}
\usepackage{times}
\usepackage{amsfonts}

\newtheorem{theorem}{Theorem}[section]
\newtheorem{lemm}[theorem]{Lemma}
\newtheorem{prop}[theorem]{Proposition}

\theoremstyle{definition}

\newtheorem{remark}[theorem]{Remark}

\renewcommand{\thefootnote}

\begin{document}

\title[Kostant's generating functions and McKay-Slodowy correspondence]
{Kostant's generating functions and McKay-Slodowy correspondence}

\author[Jing]{Naihuan Jing}
\address{Department of Mathematics, North Carolina State University,
   Raleigh, NC 27695, USA}
\email{jing@ncsu.edu}

\author[Li]{Zhijun Li}
\address{Department of Mathematics, Huzhou University, Huzhou, Zhejiang 313000, China}
\email{zhijun1010@163.com}

\author[Wang]{Danxia Wang}
\address{Department of Mathematics, Huzhou University, Huzhou, Zhejiang 313000, China}
\email{dxwangmath@126.com}

\begin{abstract}
Let $N\unlhd G$ be a pair of finite subgroups of $\mathrm{SL}_2(\mathbb{C})$ and
$V$ a finite-dimensional fundamental $G$-module. We study Kostant's generating functions for
the decomposition of the $\mathrm{SL}_2(\mathbb C)$-module $S^k(V)$
restricted to $N\lhd G$ in connection with the McKay-Slodowy correspondence.
In particular, the classical Kostant formula was generalized to a uniform version of the Poincar\'{e} series for the symmetric invariants
in which the multiplicities of any individual module
in the symmetric algebra are completely determined.
\end{abstract}
\date{}
\maketitle

\footnote{Corresponding author: dxwangmath@126.com}

\section{Introduction}

Let $G$ be a finite subgroup of the special linear group $\mathrm{SL}_2(\mathbb C)$.
The McKay correspondence \cite{Mc} gives a bijective map between finite subgroups of
$\mathrm{SL}_2(\mathbb{C})$ and affine Dynkin diagrams of simply laced $ADE$ types.
It is known that this correspondence establishes a classification of resolutions of singularities of $\mathbb C^2/G$, where
$G$ is a finite subgroup of $\mathrm{SL}_2(\mathbb C)$.
Let $N$ be a normal subgroup of $G$,
Slodowy \cite{Sl} considered a more general minimal resolution of the singularity of $\mathbb C^2/ {N}$ under the action of
$G/N$. The algebraic counterpart is the so-called  McKay-Slodowy correspondence which matches all affine Dynkin diagrams
with the pairs $N\lhd G\leq \mathrm{SL}_2(\mathbb C)$. An elementary proof of the McKay-Slodowy correspondence is included in \cite{JWZ1}.

Let $V=\mathbb C^2$ be the fundamental module of $\mathrm{SL}_2(\mathbb C)$.
The set of the $k$th symmetric tensor space
$\{S^k(V)\}$ generates all irreducible finite dimensional
representations of $\mathrm{SL}_2(\mathbb C)$ under the restriction functor according to the Schur-Weyl duality.
It is clear that the symmetric algebra $S(V)=\bigoplus_{k=0}^{\infty}S^k(V)$
is an infinite dimensional $\mathrm{SL}_2(\mathbb C)$-module.
The interesting question of decomposing  the restriction
$S^k(V)|_{G}$ into simple $G$-modules has been studied
in \cite{G-SV,Kn,Kos2,Sl,Sl2}, which have shown that the answer relied upon the McKay correspondence.
In particular, the Poincar\'e series for the special node has a simple form determined by the global data of the finite group $G$
associated to the simple Lie algebra in the McKay correspondence.
Furthermore, the Poincar\'{e} series for the multiplicity of the simple $G$-module in the
symmetric algebra $S(V)$ has been investigated in various works \cite{Kos3,Sp1,Sp2,St,Su}.

Let $\tilde{\mathfrak{g}}$ be an affine Lie algebra realized by a pair of subgroups
$N\lhd G\leq \mathrm{SL}_2(\mathbb C)$  via the McKay-Slodowy correspondence (either in the restricted quiver or
induced quiver),
and $\tilde{\mathfrak{h}}$ be the corresponding  Cartan subalgebra of $\tilde{\mathfrak{g}}$ with dimension $l+1$.
Let $\tilde{\mathfrak{h}}'$ be the dual space to $\tilde{\mathfrak{h}}$ and $\{\alpha_0,\alpha_1\ldots,\alpha_l\}
\subset \tilde{\mathfrak{h}}'$ be a set of simple roots. %with respect to some choice of a positive root system.
It is clear that the $l$-dimension subspace $\mathfrak{h}\subset\tilde{\mathfrak{h}}$ is a Cartan subalgebra of a complex simple Lie algebra $\mathfrak{g}$, denoted the dual space by $\mathfrak{h}'$ and $\{\alpha_1,\ldots,\alpha_l\}\subset \mathfrak{h}'$ is an set of
simple roots of $\mathfrak{g}$. In terms of Dynkin diagrams, the Lie algebra $\frak g$ corresponds to the sub-diagram $\Gamma$ of $\tilde{\Gamma}$
by removing the special vertex corresponding to $\alpha_0$.

In this paper we focus on the general question of
how the $\mathrm{SL}_2(\mathbb C)$-module $S^k(V)$ decomposes by
restricting to pairs of subgroups $N\lhd G\leq \mathrm{SL}_2(\mathbb C)$ in connection
with the McKay-Slodowy correspondence. Through the McKay-Slodowy correspondence, the $\mathrm{SL}_2(\mathbb C)$-module $S^k(V)$
restricted to the pair of subgroups $N\lhd G$ will
correspond to a pair of vectors in the dual space $\tilde{\mathfrak{h}}'$ of the Cartan subalgebra $\tilde{\mathfrak{h}}$ of the
multiply laced affine Lie algebras $\tilde{\mathfrak{g}}$.
We will show that Kostant's generating functions of the $G$-invariants and
$N$-invariants for each individual module corresponding to the root vertex
in the affine Dynkin diagram can be uniformly computed by rational functions determined by the data of the
pair $N\lhd G$, generalizing Kostant's original result for the simply laced case.

The paper is organized as follows. In Section $2$, we develop
several results for the multiply laced affine Lie algebras, which generalized the simply laced types by the McKay-Slodowy correspondence.
In Section $3$, Kostant's vectors and Kostant's generating functions
are described based on the McKay-Slodowy correspondence.
Moreover, the Konstant vectors are obtained by the orbit of the affine vertex $\alpha_0$ of the affine Lie algebra
under the action of corresponding Coxeter transformation.
In addition, we derive a unified formula of Kostant's generating functions for multiply laced and simply laced affine Lie algebras.
In Section $4$, we generalize a uniform formula of the  Poincar\'{e} series for the multiplicities of the irreducible module,
restricted module or induced module in the symmetric algebra.
Namely, the Poincar\'{e} series for the symmetric invariants and the multiplicities of any individual module
in the symmetric algebra can be obtained by the classical Kostant formula.

\section{The generalizations by McKay-Slodowy correspondence}

Let $N$ be a normal subgroup of the finite group $G$ and
$\{\rho_i|i\in {\rm I}_{G}\}$ (resp. $\{\phi_i|i\in {\rm I}_{N}\}$) the set of
complex finite-dimensional irreducible modules of $G$ (resp. $N$).
Let $\{\check\rho_i|i\in \check{\mathrm{I}}\}$ be the set of inequivalent $N$-restriction modules ${\rm Res}(\rho_i)$.
Correspondingly, the set $\{\hat\phi_i|i\in \rm\hat{{I}}\}$ denotes that of inequivalent induced $G$-modules
${\rm Ind}(\phi_i)$. It is known that the set ${\rm \check{I}}$ is one-to-one correspondent with the set ${\rm \hat{I}}$ (see \cite{JWZ1}).

Let $V$ be a fixed finite-dimensional $G$-module, denote the $N$-restriction
of $V$ by $\check{V}$.
Therefore the tensor decompositions
\begin{equation*}\label{resandindtensor}
  \check{V}\bigotimes \check\rho_j=\bigoplus\limits_{i\in\check{\rm I}}b_{ji} \check\rho_{i}\qquad {\rm and}\qquad
  V\bigotimes \hat\phi_j=\bigoplus\limits_{i\in\hat{\rm I}}d_{ji} \hat\phi_i
\end{equation*}
give rise to two integral matrices ${\rm \widetilde{B}}=(b_{ij})$ and ${\rm \widetilde{D}}=(d_{ij})$ of the same size respectively.
The corresponding representation graph is the digraph
$\mathcal{R}_{\check{V}}(\check{G})$ (resp. $\mathcal{R}_{V}(\hat{N})$) with vertices indexed by ${\rm \check{I}}$ (resp. ${\rm \hat{I}}$),
%indexed by $i$ ($i\in\Check{\mathrm{I}}$) as vertices,
where $i$ is joined to $j$ by $\mathrm{max}(b_{ij}, b_{ji})$ (resp. $\mathrm{max}(d_{ij}, d_{ji})$) edges with an arrow pointing to
$i$ if $b_{ij}>1$ (resp. $d_{ij}>1$).

Now let $G$ be a subgroup of $\mathrm{SL}_2(\mathbb C)$. It is well-known that there are five types:
(1) the cyclic group $C_{n+1}$; (2) the dihedral group $D_{2n+2}$; (3) the binary tetrahedral ${T}$ of order $24$,
(4) the binary octahedral group ${O}$ of order $48$; and (5)
the  binary icosahedral group ${I}$ of order $120$.

%If $N=G$ is a finite subgroup,
%the tensor product between $G$-module $V$ and an
%irreducible $G$-module $\rho_j$ ($j\in {\rm I}_{G}$) comes down to
%\begin{equation*}\label{irrtensor}
% V \bigotimes \rho_j = \bigoplus_{i\in {\rm I}_{G}} a_{ji} \rho_i,
%\end{equation*}
%where  $a_{ij}$ form a matrix ${\rm \widetilde{A}}=(a_{ij})$ and the corresponding representation graph is the digraph $\mathcal{R}_{V}(G)$,
%whose index set is ${\rm I}_{G}$ and $a_{ij}$ edges from $i$ to $j$.
%
%
%
%It is well known that the isomorphism classes of
%finite subgroups of the special linear group $\mathrm{SL}_2(\mathbb{C})$  are
%a cyclic group ${C}_n$ of order $n$, a binary dihedral group ${D}_{n}$ of order $4n$ and
%three exceptional polyhedral groups: the binary tetrahedral ${T}$ of order $24$,
%the binary octahedral group ${O}$ of order $48$, and
%the  binary icosahedral group ${I}$ of order $120$.

The McKay-Slodowy correspondence says that the distinguished pairs $N\lhd G\leq \mathrm{SL}_2(\mathbb C)$ and their
digraphs realize
all affine Dynkin diagrams explicitly as follows \cite{JWZ1}. The five types of the simply laced affine Dynkin diagrams $A_{n-1}^{(1)}$, $D_{n+2}^{(1)}$, $E_6^{(1)}$, $E_7^{(1)}$, $E_8^{(1)}$ correspond to the
special pairs of $N=G$ consisting of ${C}_n$, ${D}_{n}$, ${T}$, ${O}$, ${I}$ and $V\cong\mathbb{C}^2$
via $\mathcal{R}_{V}(G)$. The Cartan matrices are given by ${\rm C_{\widetilde{A}}=2I}-{\rm \widetilde{A}}$, where  ${\rm \widetilde{A}}$ is the adjacency matrix of
$\mathcal{R}_{V}(G)$.

The group pairs $ D_{n-1} \lhd  D_{2(n-1)}$,
$ C_{2n} \lhd  D_n$, $ C_{2n} \lhd  D_{2n}$, $ T \lhd  O$,
$ D_2 \lhd  T$, $ C_2 \lhd  D_2$  %and $V\cong\mathbb{C}^2$
 realize
the twisted %multiply laced affine
Dynkin diagrams of types
$A_{2n-1}^{(2)}$, $D_{n+1}^{(2)}$, $A_{2n}^{(2)}$, $E_6^{(2)}$, $D_4^{(3)}$, $A_2^{(2)}$ by the graphs $\mathcal{R}_{\check{V}}(\check{G})$ respectively, while
%and the diagram $\mathcal{R}_{V}({\hat{N}})$ realizes
the untwisted multiply-laced affine Dynkin diagrams of types
$B_n^{(1)}$, $C_n^{(1)}$, $C_n^{(1)}$, $F_4^{(1)}$, $G_2^{(1)}$, $A_1^{(1)}$ are realized by $\mathcal{R}_{V}({\hat{N}})$ respectively.
Moreover, their Cartan matrices are given by ${\rm C_{\widetilde{B}}=2I}-{\rm \widetilde{B}}^T$ and
${\rm C_{\widetilde{D}}}={\rm 2I}-{\rm \widetilde{D}}^T$ respectively, here
%are the corresponding Cartan matrices of twisted and
%non-twisted multiply laced affine Lie algebras,
${\rm \widetilde{B}}$ and ${\rm \widetilde{D}}$ are the adjacent matrices
of $\mathcal{R}_{V}({\check{G}})$ and $\mathcal{R}_{V}({\hat{N}})$ respectively.

Let  $\psi$ be the highest positive root of a complex simple Lie algebra $\mathfrak{g}$. Let $\{ \alpha_i\}$ be the set of simple
roots of the affine Kac-Moody algebra in types $X^{(r)}$. %{\color{red}and we index the simple roots mostly the same as in Kac \cite{K} except in type $A_{2n}^{(2)}$ the ordering is reversed}.
Then $\alpha_0+\psi=\sum\limits_{i=0}^ld_i\alpha_i$, where $d_i$ are the Kac symbols appearing in the
corresponding affine Dynkin diagrams. % {\color{red}while the Kac symbols of $A_{2n}^{(2)}$ are reversed}.
The coefficients $d_i$ can be realized
as dimensions of the irreducible representations of subgroups of $\mathrm{SL}_2(\mathbb{C})$
via the McKay-Slodowy correspondence.

%\begin{lemm}\label{dimandcoefficients}
%Let $N\lhd G$ be a pair of finite subgroups in $\mathrm{SL}_2(\mathbb{C})$, and $\{\check\rho_j|j\in {\rm\check{I}}\}$ and $\{\hat\phi_j|j\in\hat{\mathrm{I}}\}$ the sets of the $N$-restriction modules
%and the induced $G$-modules respectively. Then
%\begin{equation}\label{dims}
%  {\rm dim}\check\rho_j=d_j \ \ {\rm and} \ \ {\rm dim}\hat\phi_j=|G:N|d_j,
%\end{equation}
%where $d_j$ are the KAc symbols of the affine Kac-Moody algebra
%(except reordering in type $A_{2n}^{(2)}$.
%%integers appearing in the non-twisted and twisted multiply laced affine Dynkin diagrams
%%if $N\lhd G$ is $ D_{n-1} \lhd  D_{2(n-1)}$, $ C_{2n} \lhd  D_n$,  $ T \lhd  O$, $ D_2 \lhd  T$  respectively, and
%%$d_j$ are the integers appearing in the twisted and non-twisted multiply laced affine Dynkin diagrams
%%if $N\lhd G$ is $C_{2n} \lhd  D_{2n}$ $(n\geq 2)$, $C_2 \lhd  D_2$ respectively.
%In particular, when $N=G$ is one of the subgroups of $\mathrm{SL}_2(\mathbb{C})$, and
% $\{\rho_j|j\in {I}_{G}\}$ the set of irreducible  $G$-modules, then
%\begin{center}
%  ${\rm dim}\rho_j=d_j$
%\end{center}
%are the Kac symbols in the simply laced affine Dynkin diagram associated with $G$.
%%where $d_j$ are the integers appearing in the simply laced affine Dynkin diagrams.
%\end{lemm}

\begin{lemm}\label{dimandcoefficients}
Let $N\lhd G$ be a pair of finite subgroups in $\mathrm{SL}_2(\mathbb{C})$, and $\{\check\rho_j|j\in {\rm\check{I}}\}$ and $\{\hat\phi_j|j\in\hat{\mathrm{I}}\}$ the sets of the $N$-restriction modules and the induced $G$-modules respectively. Then
\begin{equation}\label{dims}
  {\rm dim}\check\rho_j=d_j \ \ {\rm and} \ \ {\rm dim}\hat\phi_j=|G:N|d_j,
\end{equation}
where $d_j$ are the Kac symbols of the untwisted and twisted multiply laced affine Dynkin diagrams if $N\lhd G$ is $ D_{n-1} \lhd  D_{2(n-1)}$, $ C_{2n} \lhd  D_n$,  $ T \lhd  O$, or $ D_2 \lhd  T$  respectively,
$d_j$ are the Kac symbols of the untwisted multiply laced affine Dynkin diagrams
if $N\lhd G$ is $C_2 \lhd  D_2$,
$d_j$ are the Kac symbols of the twisted multiply laced affine Dynkin diagrams
if $N\lhd G$ is $C_{2n} \lhd  D_{2n}$ $(n\geq 2)$ or $C_2 \lhd  D_2$ respectively and
the ordering of $d_j$ are reversed.
In particular, when $N=G$ is one of the subgroups of $\mathrm{SL}_2(\mathbb{C})$, and
 $\{\rho_j|j\in {I}_{G}\}$ the set of irreducible $G$-modules, then
\begin{center}
  ${\rm dim}\rho_j=d_j$
\end{center}
are the Kac symbols in the simply laced affine Dynkin diagram associated with $G$.
\end{lemm}

The coefficients of the highest root of the simply-laced affine Lie algebras satisfy a cubic polynomial equation \cite{Bur}
in the context of the McKay correspondence.
Based on Lemma \ref{dimandcoefficients} we can generalize the polynomial equation for all affine Lie algebras.

\begin{theorem}
Let $N\unlhd G$ be a pair of finite subgroups in $\mathrm{SL}_2(\mathbb{C})$, $d_i(i=0,\ldots, l)$
the Kac symbols of the untwisted and twisted multiply laced affine Dynkin diagrams, then
\begin{equation}\label{relation}
   d^2\sum_{i=0}^ld_i-3d\sum_{i=0}^ld_i^2+2\sum_{i=0}^ld_i^3=0,
\end{equation}
where $d={\rm Max}\{d_i:0\leq i\leq l\}.$
\end{theorem}

Let $h$ be the Coxeter number of a complex simple Lie algebra $\mathfrak{g}$, the next
result describes the relation
between the  Coxeter number and the dimensions of irreducible modules of
the finite groups $N\unlhd G\leq \mathrm{SL}_2(\mathbb{C})$.
\begin{prop}
Let $N\lhd G\leq \mathrm{SL}_2(\mathbb{C})$ be a distinguished pair of subgroups, and let
$\check{\rho}_i$ be the $N$-restriction module of the irreducible $G$-module $\rho_i$.
Then the Coxeter number of %(multiply laced)
the Lie algebra $\mathfrak{g}$ is
\begin{equation*}
h=\sum_{i\in \Upsilon\cap N}{\rm dim}\check{\rho}_i,
\end{equation*}
where $\Upsilon$ is a set of representative of conjugacy classes of $G$.
\end{prop}

\section{Kostant's generating functions}

Let $N \unlhd G$ be a fixed pair of finite subgroups of $\mathrm{SL}_2(\mathbb{C})$
that corresponds to an affine Lie algebra $\tilde{\mathfrak{g}}$ of type $\widetilde{X}$ via the McKay-Slodowy correspondence
\cite{JWZ1}.
Let $\{\alpha_i\}$ be the simple roots of $\tilde{\mathfrak{g}}$.
Let $V\cong\mathbb C^2$ be the standard $\mathrm{SL}_2$-module.
%It is well-known that the $k$th symmetric tensors
%$\{S^k(V)\}$ realize/contain all irreducible finite dimensional
%representations of $\mathrm{SL}_2$ as submodules.
%under the restriction functor according to the Schur-Weyl duality,
The symmetric tensor algebra $S(V)=\bigoplus_{k=0}^{\infty}S^k(V)$
is an infinite dimensional $\mathrm{SL}_2(\mathbb C)$-module, which contains all
irreducible $\mathrm{SL}_2(\mathbb C)$-modules as submodules.

Let $\check{s}^{j}_{k}$ $(j\in {\rm \check{I}})$ (resp. $\hat{s}^{j}_{k}$ ($j\in {\rm \hat{I}}$)) be the multiplicity
of the $N$-restriction $\check{\rho}_j$ in the $k$th symmetric power $S^{k}(V)$ (resp. the $G$-restriction $\hat{\phi}_j$ in $S^{k}(V)$):
\begin{equation}\label{restrict}
S^{k}(V)\big|_{N}=\sum_{j\in {\rm \check{I}}}\check{s}^{j}_{k}\check{\rho}_j \ \ \
{\rm and} \ \ \ S^{k}(V)\big|_{G}=\sum_{j\in {\rm \hat{I}}}\hat{s}^{j}_{k}\hat{\phi}_j.
\end{equation}
By the McKay-Slodowy correspondence, the corresponding vectors in the dual space $\tilde{\mathfrak{h}}'$ of the Cartan subalgebra $\tilde{\mathfrak{h}}$ of
$\tilde{\mathfrak{g}}$ are defined as
\begin{equation}\label{vectors}
\check{\textbf{s}}_k =\sum_{j\in {\rm \check{I}}}\check{s}^{j}_{k}\alpha_j  \ \ \
{\rm and/or} \ \ \ \hat{\textbf{s}}_k  =\sum_{j\in {\rm \hat{I}}}\hat{s}^{j}_{k}\alpha_j
\end{equation}
according to the type of $\tilde{\mathfrak{g}}$ (see the detailed description of realization of the Dynkin diagram of $\tilde{\mathfrak{g}}$ in \cite{JWZ1}).
In addition, we define the Kostant-type generating functions by
\begin{equation}\label{powerseries}
  \check{\textbf{S}}(t)=\sum_{k=0}^{\infty}\check{\textbf{s}}_kt^k  \ \ \
{\rm and/or} \ \ \    \hat{\textbf{S}}(t)=\sum_{k=0}^{\infty}\hat{\textbf{s}}_kt^k
\end{equation}
according to the type of $\tilde{\mathfrak{g}}$.

In particular, when $N=G$, $\check{\textbf{s}}_k=\hat{\textbf{s}}_k=\textbf{s}_k$, the
multiplicity of an irreducible $G$-module $\rho_j$ in the $k$th symmetric power $S^{k}(V)$ for each $j\in {\rm I_{G}}$, then the two
Kostant-type generating functions coincide to become the Kostant generating function
\begin{equation*}
  \textbf{S}(t)=\sum_{k=0}^{\infty}\textbf{s}_kt^k.
\end{equation*}

Kostant \cite{Kos2} has shown that the generating function $\textbf{S}(t)$ is a rational function over $\mathbb Z$.
%a quotient of explicit finite sums.
We aim to determine
$\check{\textbf{S}}(t)$ and $\hat{\textbf{S}}(t)$ using the McKay-Slodowy correspondence.
%In this section we state the unified results for the Kostant generating functions.

Let $\textbf{x}_k$ denote either of the Kostant vector in \eqref{vectors}, %or \eqref{vector},
similarly let $\textbf{X}(t)$ be the corresponding Kostant-type generating function in either of \eqref{powerseries}, %or \eqref{powerserie},
and ${\rm \widetilde{X}}$ the corresponding adjacency matrix of $\mathcal{R}_{V}({\check{G}})$ (resp.
$\mathcal{R}_{V}({\hat{N}})$ or $\mathcal{R}_{V}(G)$), so ${\rm \widetilde{X}}=2I-C^T_{\widetilde{X}}$, where
$C_{\widetilde{X}}$ is the Cartan matrix. It is clear that $\textbf{x}_k$ is the vector in the root lattice of $\tilde{\mathfrak{g}}$.
%which is the $\mathbb{Z}$-span of the corresponding simple positive roots for $k\in \mathbb{Z}_{+}$.
For each $k\in \mathbb{Z}_+$,
%Similar to the deduction of Konstant in \cite{Kos2}, let $N\unlhd G$ be a pair of finite subgroups in $\mathrm{SL}_2(\mathbb{C})$
%and  $V$ be the standard module,
the Clebsch-Gordon formula implies that
\begin{equation*}
S^{k}(V)\otimes V= S^{k+1}(V) +S^{k-1}(V),
\end{equation*}
therefore
\begin{equation*}
  \left(1-(t{\rm \widetilde{X}}-t^2)\right)\textbf{X}(t)=\alpha_0,
\end{equation*}
where $\alpha_0\in \tilde{\mathfrak{h}}'$  is the special affine root of $\tilde{\mathfrak{g}}$.
Subsequently %Since $1-(t{\rm \widetilde{X}}-t^2)$ is invertible, we see that %implies that the Kostant generating function of the power series is
\begin{equation*}
 \textbf{X}(t)=\sum_{k=0}^{\infty}(t{\rm \widetilde{X}}-t^2)^k\alpha_0,
\end{equation*}
which means that Kostant's vector $\textbf{x}_k$ are explicitly determined and we have
the following result. %that the $k$-th vector of the Kostant generating function $\textbf{X}(t)$ can be written directly.
\begin{lemm}\label{vector&matrix}
Let $N\unlhd G$ be a pair of finite subgroups in $\mathrm{SL}_2(\mathbb{C})$ corresponding to the affine type
$\widetilde{X}$ via the McKay-Slodowy correspondence. Let
$\textbf{x}_k$ be
Kostant's vector for the affine Lie algebra $\tilde{\mathfrak{g}}$, then
%one of the vectors corresponding to the decomposition of the $k$-th irreducible
%$\mathrm{SL}_2(\mathbb{C})$-module restricted to the corresponding group.
%Assume matrix ${\rm \widetilde{X}}$ be afforded by the tensor product of
%the standard $\mathrm{SL}_2(\mathbb{C})$-module with the $N$-restriction modules of irreducible $G$-modules
%(resp. the induced $G$-modules of irreducible $N$-modules or the irreducible $G$-modules).
%Then
\begin{equation}\label{xk}
  \textbf{x}_k =\left(\sum_{j=0}^{[k/2]}(-1)^j\binom{k-j}{j}(2I-C^T_{\rm \widetilde{X}})^{k-2j}\right)\alpha_0,
\end{equation}
where $C_{\rm \widetilde{X}}$  is the Cartan matrix and $\alpha_0$ is the special affine root of the affine Lie algebra $\tilde{\mathfrak{g}}$.
\end{lemm}

\subsection{Kostant's vectors and the affine Coxeter elements}

Let $\widetilde{W}$ be the affine Weyl group of the affine Lie algebra $\tilde{\mathfrak{g}}$
of type $\widetilde{X}$ generated by the reflections $s_0, s_1,\ldots, s_l$ associated the simple roots
$\{\alpha_j\}$ by $$s_i\alpha_j=\alpha_j-A_{ij}\alpha_i,$$
where $C_{\widetilde{X}}=(A_{ij})$ is the affine Cartan matrix of $\tilde{\mathfrak{g}}$.
Let $\mathcal{C}_{a}$ be the corresponding affine Coxeter element of $\tilde{\mathfrak{g}}$.
Similarly, the Weyl group $W$ of the simple Lie algebra $\mathfrak{g}$ is the subgroup generated by
the simple reflections $s_1,\ldots, s_l$, and denote the corresponding Coxeter element of $W$ by $\mathcal{C}$.

%Let $P=\{1,\ldots,l\}$ and $\widetilde{P}=P\cup\{0\}$.
The set $\widetilde{\Pi}$ of simple roots can be partitioned into three subsets:
$\{\alpha_0\}$, $\Pi_1$ and $\Pi_2$ so that the simple roots of $\Pi_j$ $(j=1,2)$
consists of simple reflections of mutually orthogonal roots. Let $P_j$ denote the index set of $\Pi_j$.
%=\{\alpha_r|r\in P_j\}$. Write
%$P_1=\{i_1,i_2,\dots, i_k\}$ and $P_2=\{i_{k+1},i_{k+2},\dots, i_l\}$,
Define
\begin{equation}\label{C_i}
  \mathcal{C}_j=\prod_{r\in P_j}s_{r}\in W.
\end{equation}
Then $\mathcal{C}_1$ and $\mathcal{C}_2$ are orthogonal involutions:
${\mathcal{C}_1}^2={\mathcal{C}_2}^2=e$, where $e$ is the identity of $\widetilde{W}$.
Clearly $\mathcal{C}=\mathcal{C}_2\mathcal{C}_1$ is the Coxeter element of $W$. As $\alpha_0$ is orthogonal to
all finite simple roots, let $\widetilde{\Pi}_1=\Pi_1$  and $\widetilde{\Pi}_2=\Pi_2\cup\{\alpha_0\}$.
Also let $\widetilde{P}_1=P_1$ and $\widetilde{P}_2=P_2\cup\{ 0\}$.

%It is clear that $\alpha_0$ is orthogonal to all the elements in $\Pi_1$ or $\Pi_2$.
%Fix that $\alpha_0$ is orthogonal to all the elements of $\Pi_2$.

The follwing result will fix the zero component of Kostant's vector $\textbf{x}_k$
for the affine Lie algebra $\hat{\mathfrak{g}}$.

 %The statement of this prop is strange, unclear!
 \begin{prop}
Let $N\unlhd G\leq \mathrm{SL}_2(\mathbb{C})$ be the pair realizing the affine Lie algebra $\hat{\mathfrak{g}}$ in the McKay-Slodowy correspondence.
Let
$\textbf{x}_k$ be  Kostant's vector for the affine Lie algebra $\hat{\mathfrak{g}}$.
Then the multiplicities of the corresponding roots in the vector $\textbf{x}_k$ are
%module in the $k$th symmetric power $S^{k}(V)$ is
\begin{equation*}
  \check{s}^{i}_{k}=\hat{s}^{i}_{k}=0
\end{equation*}
 for all $i\in \widetilde{P}_j$,
where $j\in\{1,2\}$ and  $j-1\equiv k \  {\rm mod} \  2$.
\end{prop}
\begin{proof}
We only show the case of the $N$-restriction, the other case is similar.
Note that the decomposition
$S^{k}(V)\big|_{N}=\sum_{i=0}^{l}\check{s}^{i}_{k}\check{\rho}_i$ means that
the corresponding character value at $-1\in \mathrm{SL}_2$ is
\begin{equation*}
\chi_{S^{k}(V)}(-1)=\sum_{i=0}^{l}\check{s}^{i}_{k}\chi_{\check{\rho}_i}(-1)=(-1)^k.
\end{equation*}
The above formula gives rise to $\check{s}^{i}_{k}=0$  for all $i\in \widetilde{P}_j$,
where $j\in\{1,2\}$ and  $j-1\equiv k \  {\rm mod} \  2$.

\end{proof}

Let $\mathcal{C}_{a_1}=\mathcal{C}_1$ and $\mathcal{C}_{a_2}=s_0\mathcal{C}_2=\mathcal{C}_2 s_0$,
where $s_0$ is the reflection corresponding to $\alpha_0$ for the affine Lie algebra $\hat{\mathfrak{g}}$.
It is clear that ${\mathcal{C}_{a_1}}^2={\mathcal{C}_{a_2}}^2=e$ and
$\mathcal{C}_{a}=\mathcal{C}_{a_2}\mathcal{C}_{a_1}=s_0\mathcal{C}_2\mathcal{C}_1=s_0\mathcal{C}$ is  the Coxeter transformation in $\widetilde{W}$.
The adjacency matrix of $\hat{\mathfrak{g}}$ can be expressed as
\begin{center}
${\rm \widetilde{X}}=\mathcal{C}_{a_1}+\mathcal{C}_{a_2},$
\end{center}
which has been investigated by different methods in \cite{BLM,Kos2,Ste}.

Kostant has shown that the vector $\textbf{s}_k$ can be
obtained by the affine Coxeter transformation via the adjacency matrix of the simply laced affine Lie algebras.
%and the corresponding Coxeter transformation.
In the following we generalize this result to arbitrary affine Lie algebra and it turns out a uniformed formula can be found for all cases.

The square of the adjacency matrix ${\rm \widetilde{X}}$ of $\tilde{\mathfrak{g}}$ can be expressed as
\begin{eqnarray*}
  {\rm \widetilde{X}}^{2}
&=& (\mathcal{C}_{a}+I)+({\mathcal{C}_{a}}^{-1}+I) \\
   &=& (\mathcal{C}_{a}+I)({\mathcal{C}_{a}}^{-1}+I) \\
   &=& (\mathcal{C}_{a}+I)^2{\mathcal{C}_{a}}^{-1}.
\end{eqnarray*}
It then follows for $m\in \mathbb{Z}_{+}$
\begin{equation}\label{adjmatrix}
  {\rm \widetilde{X}}^{2m}=\sum_{i=0}^{2m}\binom{2m}{i}{\mathcal{C}_{a}}^{i-m}.
\end{equation}

Combining with Lemma \ref{vector&matrix} we see that Kostant's vector
\begin{align}\label{e:xk1}
  \textbf{x}_k &= \left(\sum_{j=0}^{[k/2]}(-1)^j\binom{k-j}{j}{\rm \widetilde{X}}^{k-2j}\right)\alpha_0 \\ \label{e:xk}
   &=\sum_{j=0}^{m}(-1)^j\binom{k-j}{j}\sum_{i=0}^{2(m-j)}\binom{2(m-j)}{i}{\mathcal{C}_{a}}^{i+j-m}
\left\{
  \begin{array}{ll}
    \alpha_0 &\\
    \textbf{x}_1 &
  \end{array},
\right.
\end{align}
where the last factor is $\alpha_0$ for $k=2m$ or $\textbf{x}_1$ for $k=2m+1$.

To proceed, we need the following useful identities.
\begin{lemm}\emph{\cite[Lem. 3.11]{Kos2}}\label{combination}
Let $j,n\in \mathbb{Z}_{+}$, where $j\leq n$. Then
\begin{equation*}
  \sum_{i=0}^{j}(-1)^i\binom{n-2i}{j-i}\binom{n-i}{i}=1
\end{equation*}
and
\begin{equation*}
  \sum_{i=0}^{j}(-1)^i\binom{n-2i}{j-i}\binom{n+1-i}{i}=\left\{
                                                           \begin{array}{ll}
                                                             1, & \hbox{if \ j \ is \ even,} \\
                                                             0, & \hbox{if \ j \ is \ odd.}
                                                           \end{array}
                                                         \right.
\end{equation*}
\end{lemm}

Using Lemma \ref{combination} the expressing on the right of \eqref{e:xk1}
can be simplified, and %By combining with the relations of binomial coefficients in Lemma \ref{combination},
Kostant's vector $\textbf{x}_k$ can be written as a finite sum of powers of the Coxeter element $\mathcal{C}_{a}\in \widetilde{W}$.
In particular, for $k=2m$, $\textbf{x}_k$ is the sum of the orbit of $\alpha_0$ under the action of $\mathcal{C}_{a}$.
Summarizing the above, we have shown that

\begin{prop}\label{orbits}
Let $\textbf{x}_k$ be one of Kostant's vectors  corresponding to the decomposition of the $N$-induction or $G$-restriction of the $k$-th irreducible
$\mathrm{SL}_2(\mathbb{C})$-module. Let
$\mathcal{C}_{a}$ be the affine Coxeter transformation of the affine Lie algebra $\tilde{\mathfrak{g}}$ corresponding to
$(N, G)$.
Assume $\alpha_0$ is the special affine root. Then for $k=2m$,
\begin{equation}\label{xk1}
\textbf{x}_k=\left(\sum_{j=0}^{2m}{\mathcal{C}_{a}}^{j-m} \right)\alpha_0;
\end{equation}
and for $k=2m+1$,
\begin{equation}\label{xk2}
\textbf{x}_k=\left(\sum_{j=0}^{m}{\mathcal{C}_{a}}^{2j-m} \right)\textbf{x}_1.
\end{equation}
\end{prop}

\subsection{Formulas for Kostant-type generating functions}

In this subsection, we follow the argument of Kostant \cite{Kos2} %for the simply laced affine Lie algebras,
to derive a unified formula for Kostant-type generating function in all cases. %is determined by the notions of following construction.

Let $a_0=1$ and define
\begin{equation}\label{a_i}
  a_i=2\frac{(\mathcal{C}^i\alpha_0,\alpha_0)}{(\alpha_0,\alpha_0)}
\end{equation}
for $i\in  \mathbb{Z}_+$. %which are based on the bilinear form on the affine Lie algebra $\tilde{\mathfrak{g}}$.
Clearly $a_i$ are integers since $\mathcal{C}^i\alpha_0$ is an element of the root lattice.
The sequence is periodic since $\mathcal{C}^h=e$ for the Coxeter number $h$ of the simple Lie algebra
$\mathfrak{g}$.

Recall that the affine Coxeter transformation
$\mathcal{C}_{a}=s_0\mathcal{C}_2\mathcal{C}_1=s_0\mathcal{C}$. The first few elements of the orbit of $\alpha_0$ under the action of $\mathcal{C}_{a}$ are
\begin{align*}
  \mathcal{C}_{a}\alpha_0 =& s_0(\mathcal{C}\alpha_0)=\mathcal{C}\alpha_0-a_1\alpha_0,\\
  {\mathcal{C}_{a}}^2\alpha_0 =&\mathcal{C}^2\alpha_0-a_1\mathcal{C}\alpha_0+(-a_2+a_1^2)\alpha_0,\\
  {\mathcal{C}_{a}}^3\alpha_0 =&\mathcal{C}^3\alpha_0-a_1\mathcal{C}^2\alpha_0+(-a_2+a_1^2)\mathcal{C}\alpha_0+
(-a_3+2a_1a_2-a_1^3)\alpha_0,\\
  {\mathcal{C}_{a}}^4\alpha_0 =&\mathcal{C}^4\alpha_0-a_1\mathcal{C}^3\alpha_0+(-a_2+a_1^2)\mathcal{C}^2\alpha_0
       +(-a_3+2a_1a_2-a_1^3)\mathcal{C}\alpha_0\\
   & +(-a_4+2a_1a_3-3a_1^2a_2+a_2^2+a_1^4)\alpha_0.
\end{align*}
The elements $\mathcal{C}^i\alpha_0$ can be iteratively computed as follows. Suppose
\begin{equation}\label{e:Ca}
{\mathcal{C}_a}^i\alpha_0=\sum_{j=0}^ib_j\mathcal C^{i-j}\alpha_0.
\end{equation}
Then
\begin{align*}
  {\mathcal{C}_{a}}^{i+1}\alpha_0 &= s_0 \mathcal C{\mathcal C_a}^i\alpha_0
  =\mathcal C{\mathcal C_a}^i\alpha_0-\langle \mathcal C{\mathcal C_a}^i\alpha_0, \alpha_0^{\vee}\rangle\alpha_0\\
  &= \sum_{j=0}^ib_j\mathcal C^{i+1-j}\alpha_0-\sum_{j=0}^ib_ja_{i+1-j}\alpha_0.
\end{align*}
Therefore the coefficients in \eqref{e:Ca} are given inductively by
\begin{equation}\label{b_i}
  b_0=1 \ \ {\rm and} \ \ b_s=-\sum_{j=0}^{i-1}b_ja_{s-j}, \ \ s\in \mathbb{Z}_+.
\end{equation}
Alternatively, define the generating series $a(t)=\sum_{i=0}^{\infty}a_it^i$ for \eqref{a_i}. Then the generating series $b(t)=\sum_{i=0}^{\infty}b_it^i$ for the coefficients of \eqref{e:Ca}
is actually the reciprocal of $a(t)$, i.e.
\begin{equation}\label{inverse}
a(t)b(t)=1.
 % \sum_{i=0}^{\infty}a_it^i\sum_{s=0}^{\infty}b_st^s=1,
\end{equation}
Note that $a(t)$ is a rational function due to periodicity of the $a_i$:  %is established by \eqref{b_i}.
\begin{align}
a(t)=\frac{\sum_{i=0}^{h-1}a_it^i+t^h}{1-t^h}.
\end{align}
Therefore
\begin{align}
b(t)=\frac{1-t^h}{\sum_{i=0}^{h-1}a_it^i+t^h}.
\end{align}
Thus we have the general expression for the orbit of $\alpha_0$ under $\langle \mathcal{C}_a\rangle $ for the general affine Lie algebra
$\tilde{\mathfrak{g}}$,
generalizing Kostant's formula for the simply laced case {\rm \cite[Lem. 4.2-4.3]{Kos2}}:
%Therefore for the general affine Lie algebra including nonsimply laced cases,
%there have the orbits of the affine Coxeter transformation
%act on the affine vertex $\alpha_0$, which is same as the case of simply-laced.
%As a consequence, the following expressions are suitable to the
%simply laced and multiply laced affine Lie algebra.
\begin{lemm}\label{powerofca}
Let $\mathcal{C}_{a}$ be the affine Coxeter transformation of an affine Lie algebra with
the affine root $\alpha_0$. Then the action of ${\mathcal{C}_{a}}^i$ on $\alpha_0$ is given by
\begin{equation*}
  {\mathcal{C}_{a}}^i\alpha_0=\sum_{j=0}^ib_j\mathcal{C}^{i-j}\alpha_0
\end{equation*}
and
\begin{equation*}
  {\mathcal{C}_{a}}^{-i}\alpha_0=-\sum_{j=0}^{i-1}b_j\mathcal{C}_{a_1}\mathcal{C}^{i-j-1}\alpha_0,
\end{equation*}
where $\mathcal{C}$ is the finite Coxeter transformation, $b_j$ and $\mathcal{C}_{a_1}=\mathcal C_1$
are defined in \eqref{b_i} and \eqref{C_i} respectively.
\end{lemm}
\begin{remark}
Note that the orbit ${\mathcal{C}_{a}}^{-i}\alpha_0$ ($i\in \mathbb{Z}_+$) can be obtained  immediately
by combining the relations among ${\mathcal{C}_{a}}^i$, $\mathcal{C}_{a_1}$ and $\mathcal{C}_{a_2}$.
\end{remark}

For  $k\in \mathbb{Z}_+$ let
\begin{equation*}
  \mathcal{C}_{k}=\left\{
                      \begin{array}{ll}
                        \mathcal{C}_{1}, & \hbox{if \ $k$ \ is \ odd} \\
                        \mathcal{C}_{2}, & \hbox{if \ $k$ \ is \ even}
                      \end{array}
                    \right..
\end{equation*}
Set $\mathcal{C}^{(0)}=e$ and for $k\in \mathbb{Z}_+$ let
\begin{equation}\label{c^k}
  \mathcal{C}^{(k)}=\mathcal{C}_{k}\mathcal{C}_{k-1}\cdots\mathcal{C}_{2}\mathcal{C}_{1},
\end{equation}
namely $\mathcal{C}^{(k)}$ is an alternating product of $\mathcal{C}_{1}$ and $\mathcal{C}_{2}$ with $k$ factors.
Furthermore, ${\mathcal{C}}^{(k)}\alpha_0$ $(k\in \mathbb{Z}+)$ is a root of the affine Lie algebra $\tilde{\mathfrak{g}}$
since $\alpha_0$ is a root of $\tilde{\mathfrak{g}}$.

We can express Kostant's vector $\textbf{x}_k$ in terms of $\mathcal C^{(k)}$.
%can be given by a sum of two products.

%Let $w_0=\alpha_0$ and for  $k\in\mathbb{Z}_+$
%\begin{equation*}
%  w_k=\mathcal{C}^{(k)}\alpha_0-\mathcal{C}^{(k-1)}\alpha_0.
%\end{equation*}

\begin{theorem}\label{quo1}
Let $\textbf{x}_k$ be Kostant's vector for the affine Lie algebra $\tilde{\mathfrak{g}}$.
%he decomposition of the $k$-th irreducible
%$\mathrm{SL}_2(\mathbb{C})$-module restricted to the corresponding group.
Then
\begin{equation}\label{e:xk}
  \textbf{x}_k=\sum_{i=0}^{\lfloor k/2\rfloor}\left(\sum_{s=0}^{i}b_s\right)
\left(\mathcal{C}^{(k-2i)}\alpha_0-\mathcal{C}^{(k-2i-1)}\alpha_0\right),
\end{equation}
where $b_s$ and $\mathcal{C}^{(k)}$ are given by the equations \eqref{b_i} and \eqref{c^k} respectively.
\end{theorem}
\begin{proof} The proof is divided into two similar cases according the parity of $k$.
We assume $k=2m$ first. If follows from \eqref{xk1} and Lemma \ref{powerofca} that
%Lemma \ref{powerofca} and Prop.\eqref{xk1}
%to the Kostant vector $\textbf{x}_k$ in  of
\begin{align*}
  \textbf{x}_k%\left(\sum_{j=0}^{2m}{\mathcal{C}_{a}}^{j-m} \right)\alpha_0\\
   & =\alpha_0+\sum_{j=1}^{m}\left({\mathcal{C}_{a}}^{j}\alpha_0+{\mathcal{C}_{a}}^{-j}\alpha_0\right)\\
   & =\alpha_0+\sum_{j=1}^{m}\left[b_m\alpha_0+\sum_{s=0}^{j-1}b_s\left({\mathcal{C}}^{j-s}\alpha_0-
\mathcal{C}_{a_1}{\mathcal{C}}^{j-s-1}\alpha_0\right)\right]\\
 & =\alpha_0+\sum_{j=1}^{m}\left[b_m\alpha_0+\sum_{s=0}^{j-1}b_s\left({\mathcal{C}}^{(2j-2s)}\alpha_0-
 {\mathcal{C}}^{(2j-2s-1)}\alpha_0\right)\right]\\
%& =\alpha_0+\sum_{j=1}^{m}\left(b_m\alpha_0+\sum_{s=0}^{j-1}b_sw_{2j-2s}\right)\\
& =\alpha_0+\sum_{j=1}^{m}\sum_{s=0}^{j}b_s\left(\mathcal{C}^{(2j-2s)}\alpha_0-\mathcal{C}^{(2j-2s-1)}\alpha_0\right)\\
& =\sum_{j=0}^{m}\sum_{s=0}^{j}b_s\left(\mathcal{C}^{(2j-2s)}\alpha_0-\mathcal{C}^{(2j-2s-1)}\alpha_0\right)\\
& =\sum_{i=0}^{m}\left(\sum_{s=0}^{i}b_s\right)\left(\mathcal{C}^{(k-2i)}\alpha_0-\mathcal{C}^{(k-2i-1)}\alpha_0\right).
\end{align*}
Set $i=m-j+s$ then $2j-2s=k-2i$, replacing $i$ with $j$ gives the last equality. The case of $k=2m+1$ is treated similarly.
%Case 2, if $k=2m+1$, similar to the case of $k=2m$, the vector $\textbf{x}_k$ in \eqref{xk2} can be obtained step by step.
\end{proof}

Kostant-type generating function $\textbf{X}(t)$ in \eqref{powerseries} can then be computed by Theorem \ref{quo1}.
In fact, using \eqref{e:xk} we have
\begin{equation}\label{e:xt2}
  \textbf{X}(t)=\left(\sum_{i=0}^{\infty}\sum_{s=0}^{i}b_st^{2i}\right)
\left(\sum_{j=0}^{\infty}\left(\mathcal{C}^{(j)}\alpha_0-\mathcal{C}^{(j-1)}\alpha_0\right)t^j\right).
\end{equation}
Next we will see that $\textbf{X}(t)$ can be explicitly expressed as a rational function for all affine Lie algebras,
which generalizes Kostant's result for simply laced types.

Let $g=h/2\in \mathbb{Z}_+$, $h$ is the Coxeter number of $\mathfrak{g}$ (excluding type $A_{2n}$).
Steinberg \cite{Ste1} had shown that $\mathcal{C}^g\in W$ is the longest element of the Weyl group,
which takes the positive roots of $\mathfrak{g}$ to the negative roots, then $\mathcal{C}^g\psi=-\psi$ for the
highest positive root $\psi$.
Note that $(\mathcal{C}^g)^2=e$.
Since $\alpha_0+\psi\in \tilde{\mathfrak{h}}'$ is fixed under the action of $\widetilde{W}$, we have
\begin{center}
$\mathcal{C}^{(h)}\alpha_0=\mathcal{C}^g\alpha_0=\mathcal{C}^g[(\alpha_0+\psi)-\psi]=\alpha_0+2\psi$.
\end{center}
Note that $\alpha_0$ is orthogonal to all elements of $\Pi_2$,
\begin{equation}\label{c^{(h-1)}}
  \mathcal{C}^{(h-1)}\alpha_0=\mathcal{C}_2\mathcal{C}^{(h)}\alpha_0=\mathcal{C}_2(\alpha_0+2\psi)
=\mathcal{C}_2[2(\alpha_0+\psi)-\alpha_0]=\alpha_0+2\psi.
\end{equation}
Subsequently, $\mathcal{C}^{(0)}\alpha_0+\mathcal{C}^{(h)}\alpha_0=2(\alpha_0+\psi)$. It follows from
the definition of $\mathcal{C}^{(i)}$ that for
 $i\geq 1$
\begin{equation}\label{c^{(i)}}
  \mathcal{C}^{(i)}\alpha_0+\mathcal{C}^{(i+h)}\alpha_0=2(\alpha_0+\psi).
\end{equation}

Since $\widetilde{W}$ fixes $\alpha_0+\psi$, and also $\alpha_0=(\alpha_0+\psi)-\psi$, we see that
$a_i$ can also be rewritten as (see \eqref{a_i})
\begin{equation*}
  a_i=\frac{2(\mathcal{C}^i\psi,\psi)}{(\psi,\psi)}, \ \ i\geq 1.
\end{equation*}
Recall that $\mathcal{C}^g\psi=-\psi$, then
\begin{equation}%\label{a_i0}
a_i+a_{i+g}=0
\end{equation}
for $i\geq1$. Note that $a_0+a_g=-1$. This is because $a_g=-2$ and $a_0=1$ by the convention.

In order to provide a uniform  Kostant's generating function,
we relabel the first $g-1$ terms of $a_i$ as
\begin{equation}\label{def1}
  c_0=a_0=1, \ \ {\rm and} \ \ c_i=a_i=\frac{2(\mathcal{C}^i\psi,\psi)}{(\psi,\psi)}\ \ (for \ i=1,\ldots, g-1),
\end{equation}
in addition, define $c_g=-1$.
Then the above derivation has shown the following result.

\begin{theorem}
Let $\textbf{x}_k$ be one of Kostant's vectors  corresponding to the decomposition of the $k$-th irreducible
$\mathrm{SL}_2(\mathbb{C})$-module restricted to the corresponding group. Then Kostant-type generating function
$\textbf{X}(t)=\sum_{k=0}^{\infty}\textbf{x}_kt^k$ is given by
\begin{equation*}
  \textbf{X}(t)=\frac{\sum\limits_{j=0}^{h}\left(\mathcal{C}^{(j)}\alpha_0-\mathcal{C}^{(j-1)}\alpha_0\right)t^j}
{\sum\limits_{i=0}^{g+1}(c_i-c_{i-1})t^{2i}},
\end{equation*}
where $\mathcal{C}^{(j)}$ and $c_i$ are defined in \eqref{c^k} and \eqref{def1}, $\alpha_0$ is
the affine vertex of the affine Lie algebra $\tilde{\mathfrak g}$.
\end{theorem}
\begin{proof} It follows from \eqref{e:xt2} that
\begin{align*}\label{x(t)}
   \textbf{X}(t)
       &=\left(\sum_{i=0}^{\infty}\left(\sum_{s=0}^{i}b_s\right)t^{2i}\right)\left(\sum_{j=0}^{\infty}
\left(\mathcal{C}^{(j)}\alpha_0-\mathcal{C}^{(j-1)}\alpha_0\right)t^{j}\right)\\
       &=\left(\frac{1}{1-t^2}\sum_{i=0}^{\infty}b_it^{2i}\right)\left((1-t)\sum_{j=0}^{\infty}\mathcal{C}^{(j)}\alpha_0t^j\right)\\
       &=\frac{(1-t)\sum\limits_{j=0}^{\infty}\mathcal{C}^{(j)}\alpha_0t^j}{(1-t^2)\left(\sum\limits_{i=0}^{\infty}a_it^{2i}\right)}\\
\end{align*}
To simplify the expression, we multiply both numerator and denominator by $1+t^h$ to get
\begin{align*}
   \textbf{X}(t)   &=\frac{(1-t)\left(\sum\limits_{j=0}^{h-1}\mathcal{C}^{(j)}\alpha_0t^j+\sum\limits_{j=0}^{\infty}(\mathcal{C}^{(j)}\alpha_0+\mathcal{C}^{(j+h)}\alpha_0)t^{j+h}\right)}
{(1-t^2)\left(\sum\limits_{i=0}^{g-1}a_it^{2i}+\sum\limits_{i=0}^{\infty}(a_i+a_{i+g})t^{i+h}\right)}\\
       &=\frac{(1-t)\left(\sum\limits_{j=0}^{h-1}\mathcal{C}^{(j)}\alpha_0t^j+\frac{2(\alpha_0+\psi)t^h}{1-t}\right)}
{(1-t^2)\left(\sum\limits_{i=0}^{g-1}a_it^{2i}-t^{h}\right)}\\
 &=\frac{\sum\limits_{j=0}^{h-1}\left(\mathcal{C}^{(j)}\alpha_0-\mathcal{C}^{(j-1)}\alpha_0\right)t^j-\mathcal{C}^{(h-1)}\alpha_0t^h+2(\alpha_0+\psi)t^h}
{(1-t^2)\left(\sum\limits_{i=0}^{g}c_it^{2i}\right)}\\
&=\frac{\sum\limits_{j=0}^{h}\left(\mathcal{C}^{(j)}\alpha_0-\mathcal{C}^{(j-1)}\alpha_0\right)t^j}
{\sum\limits_{i=0}^{g+1}(c_i-c_{i-1})t^{2i}}.
\end{align*}
\end{proof}

For $i=0,1,\dots,h$, denote
\begin{equation*}
   z_i=\mathcal{C}^{(i)}\alpha_0-\mathcal{C}^{(i-1)}\alpha_0,
\end{equation*}
and put
\begin{equation*}
   z(t)=\sum_{i=0}^hz_it^i.
\end{equation*}

Then for any affine Lie algebra, we have the following result.
%$z_0=\alpha_0$. Moreover, Konstant \cite{Kos2} has shown that
%$z_0=z_h=\alpha_0$ and $z_k\in \mathfrak{h}'$ for $1\leq k\leq h-1$, which means
%that the $\alpha_0$-component of $z_k$ vanishes.

\begin{prop}\label{prop}
Let $N \lhd G\leq \mathrm{SL}_2(\mathbb{C})$ be the pair associated with the affine Lie algebra
$\tilde{\mathfrak g}$. %(resp. $N=G\leq \mathrm{SL}_2(\mathbb{C})$).
For $k=1,\dots,h-1$, the root $\mathcal{C}^{(k)}\psi$ are distinct and simply ordered with respect
to the $\mathbb{Z}_+$-cone spanned by $\Delta_+$ with $\mathcal{C}^{(0)}\psi=\psi$ the highest
and $\mathcal{C}^{(h-1)}\psi=-\psi$ the lowest, moreover, there is
$z_k\neq 0$ and
\begin{equation*}
  z_k=z_{h-k}.
\end{equation*}
Furthermore there exists $i_{\star}\in P_j$, $j=1,2$ with $j \equiv g$  {\rm mod} $2$ such that
\begin{equation*}
  \mathcal{C}^{(g-1)}\psi=\alpha_{i_{\star}}\ \  {\rm and} \ \ z_g=2\alpha_{i_{\star}},
\end{equation*}
where $\alpha_{i_{\star}}$ is the root corresponding to the unique long root connected to the short root with a
multiply laced (resp. the branch point).
\end{prop}

\begin{prop}\label{SL2}
Let $N \unlhd G\leq \mathrm{SL}_2(\mathbb{C})$ be a pair of subgroups.
Let $\check{\rho}_i$ be the $N$-restriction module of the irreducible $G$-module ${\rho}_i$,
$\phi_j$ an irreducible $N$-module, $\Upsilon$ the set of a representative of conjugacy classes of $N$.
Let
\begin{equation*}
  h=\sum_{i\in \Upsilon\cap N}{\rm dim}\check{\rho}_i,
\end{equation*}
which is the Coxeter number
of the finite  dimensional Lie algebra $\mathfrak g$ obtained from the affine Dynkin diagram by removing the
special affine vertex.
Define
\begin{align*}
a&=2{\rm max}\{ {\rm dim} \phi_i|i\in \Upsilon \}\\
b&=h+2-a.
\end{align*}
If $N \lhd G\leq \mathrm{SL}_2(\mathbb{C})$, then
\begin{equation*}
  ab=2\left|N\right|.
\end{equation*}
In particular, if $N=G\leq \mathrm{SL}_2(\mathbb{C})$, then
\begin{equation*}
  ab=2\left|G\right|.
\end{equation*}
In addition, one has
\begin{equation*}
  \sum\limits_{i=0}^{g+1}(c_i-c_{i-1})t^{2i}=(1-t^a)(1-t^b).
\end{equation*}
Furthermore, Kostant-type generating function can be expressed as
\begin{equation}\label{x(t)1}
    X(t)=\frac{z(t)}{(1-t^a)(1-t^b)}=\frac{\sum\limits_{i=0}^hz_it^i}{(1-t^a)(1-t^b)}.
\end{equation}
\end{prop}

\section{The unified Poincar\'{e} series for the Coxeter element on the root system associated to $N \unlhd G\leq \mathrm{SL}_2(\mathbb{C})$}

Let $N \unlhd G\leq \mathrm{SL}_2(\mathbb{C})$ and $V=\mathbb C^2$,
then the $i$-th component $X(t)_i$ of Kostant's generating function $X(t)$
is one of the Poincar\'{e} series for the multiplicities of the individual module $\check{\rho}_i$
(resp. $\hat{\phi}_i$ or $\rho_i$) in the symmetric algebra
$S(V)=\bigoplus\limits_{k \geq 0}S^{k}(V)$ respectively.
Formula \eqref{x(t)1} in Proposition \ref{SL2} implies that considering the Poincar\'{e} series $X(t)_i$ for the restriction
of the module $\check{\rho}_i$
(resp. $\hat{\phi}_i$ or $\rho_i$) is simply given by only
the $i$-th component $z(t)_i$ of the polynomial $z(t)$, or the coefficient of $\alpha_i$.

Let $a, b, h$ be given conceptually as in Proposition \ref{SL2},
in \cite[Thm.3.6]{JWZ2}. Then we have the generalized Kostant-type formula \cite{Kos,Kos2} for the Poincar\'{e} series
of the symmetric invariants
\begin{equation*}
  X(t)_0=\frac{1+t^h}{(1-t^a)(1-t^b)}
\end{equation*}
in terms of the McKay-Slodowy correspondence. We have also shown that
it is a uniform formula
for $N \unlhd G\leq \mathrm{SL}_2(\mathbb{C})$ (resp. $G\leq \mathrm{SL}_2(\mathbb{C})$ except for $G$ being the cyclic group of odd order)
which correspond to a pair of multiply (resp. simply) laced affine Dynkin diagrams.

In the next theorem, we will generalize Kostant's result of the Poincar\'{e} series for the individual
module $\rho_i$ of $G\leq \mathrm{SL}_2(\mathbb{C})$ to the $N$-restriction module $\check{\rho}_i$ or the
$G$-induced module $\hat{\phi}_i$ for $N \unlhd G\leq \mathrm{SL}_2(\mathbb{C})$.
In order words, there is also a unified Poincar\'{e} series for the Coxeter element for the root system
associated to $N \unlhd G\leq \mathrm{SL}_2(\mathbb{C})$.

%We summarize what we have generalization for any affine Lie algebra in the next theorem.

\begin{theorem}
Let $N \unlhd G\leq \mathrm{SL}_2(\mathbb{C})$ and $V=\mathbb C^2$. Then
the Poincar\'{e} series for the multiplicities of the module $\check{\rho}_i$
(resp. $\hat{\phi}_i$ or $\rho_i$) in the symmetric algebra
$S(V)=\bigoplus\limits_{k \geq 0}S^{k}(V)$ is given by
\begin{equation*}
  X(t)_i=\frac{z(t)_i}{(1-t^a)(1-t^b)}.
\end{equation*}
\end{theorem}

Proposition \ref{prop} leads to the following fact about the  coefficients of $z(t)_i$ for the
Poincar\'{e} series $X(t)_i$.

\begin{prop}
The sum of the coefficients of $z(t)_i$ is $2d_i$, where $d_i$ is the coefficient of the simple root $\alpha_i$ in
$\alpha_0+\psi$.  Furthermore, the coefficient
of $t^{g+k}$ is equal to the coefficient of $t^{g-k}$ for $k=0,\dots, g$ and it vanishes if $k=0$ when
$i\neq i_{\star}$.
\end{prop}

\qquad
\vskip30pt \centerline{\bf ACKNOWLEDGMENT}

The research is supported in part by
NSFC grant nos. 12171303, 12101231, 12271332 and
Huzhou University grant no. 2021XJKY03.

\end{document}